\documentclass[11pt]{article}
\date{}

\usepackage[title]{appendix}
\usepackage{xcolor}
\usepackage[margin=1in]{geometry}                
\usepackage{graphicx}
\usepackage{subcaption}
\usepackage{pdflscape}
\usepackage{amssymb}
\usepackage[normalem]{ulem}
\usepackage{hyperref}
\usepackage{enumitem}
\usepackage{mathrsfs}
\usepackage{epstopdf}
\usepackage{rotating}
\usepackage{longtable} 
\usepackage{adjustbox}
\usepackage{float}

\usepackage{color}
\usepackage{bbm, dsfont}
\usepackage{pst-node}
\usepackage{tikz-cd}
\usepackage{amsfonts} 
\usepackage{geometry}
\usepackage{amsthm}
\usepackage{amsmath}
\usepackage{titlesec}
\usepackage{amssymb}
\usepackage{enumitem}
\usepackage{float}
\usepackage [english]{babel}
\usepackage [autostyle, english = american]{csquotes}
\usepackage{algorithm}
\usepackage[noend]{algpseudocode} 
\makeatletter
\def\BState{\State\hskip-\ALG@thistlm}
\makeatother
\usepackage{hyperref}
\usepackage[normalem]{ulem}
\usepackage{mathrsfs}
\usepackage[italicdiff]{physics}

\newlist{casess}{enumerate}{1}
\setlist[casess]{label=     \textbf{Case} \arabic*:}
\usepackage{mathtools}

\makeatletter
\newcommand*{\rom}[1]{\expandafter\@slowromancap\romannumeral #1@}
\makeatother

\usepackage{etoolbox}

\makeatletter
\patchcmd{\ttlh@hang}{\parindent\z@}{\parindent\z@\leavevmode}{}{}
\patchcmd{\ttlh@hang}{\noindent}{}{}{}
\makeatother

\usepackage{listings}
\usepackage{color} 
\definecolor{mygreen}{RGB}{28,172,0} 
\definecolor{mylilas}{RGB}{170,55,241}

\newlist{Assumptions}{enumerate}{1}
\setlist[Assumptions]{label=     \textbf{Assumption} \arabic*:}

\makeatletter

\newsavebox{\@brx}
\newcommand{\llangle}[1][]{\savebox{\@brx}{\(\m@th{#1\langle}\)}%
  \mathopen{\copy\@brx\kern-0.5\wd\@brx\usebox{\@brx}}}
\newcommand{\rrangle}[1][]{\savebox{\@brx}{\(\m@th{#1\rangle}\)}%
  \mathclose{\copy\@brx\kern-0.5\wd\@brx\usebox{\@brx}}}
\makeatother

\usepackage{lipsum} 
\usepackage{titlesec}
\titleformat{\subsection}[runin]
       {\normalfont\bfseries}
       {\thesubsection}
       {0.5em}
       {}
       [.]

 \newtheorem{thm}{Theorem}[section]
 \newtheorem{cor}[thm]{Corollary}
 
 \newtheorem{lem}[thm]{Lemma}
 
 \theoremstyle{definition}
 
 \theoremstyle{remark}

 \newtheorem{ex}[thm]{Example}
 \newtheorem{que}[thm]{Question}
 \numberwithin{equation}{section}

\numberwithin{equation}{section}

\def\N{\mathbb{N}}
\def\T{\mathbb{T}}
\def\Z{\mathbb{Z}}

\def\Z{\mathbb Z}

\newcommand{\U}{\mathcal{U}}

\DeclarePairedDelimiterX{\inp}[2]{\langle}{\rangle}{#1, #2}

\makeatletter
\newcommand*\bigcdot{\mathpalette\bigcdot@{.5}}
\newcommand*\bigcdot@[2]{\mathbin{\vcenter{\hbox{\scalebox{#2}{$\m@th#1\bullet$}}}}}
\makeatother

\def\CC{\mathbb C}

\def\<{\langle}
\def\>{\rangle}
\providecommand{\norm}[1]{\lVert#1\rVert}

\numberwithin{equation}{section}

\usepackage[backend=biber,maxnames=10]{biblatex}
\begin{document}

\title{The convolution product is generally not an open map}

\author{Felipe I. Flores
}

\maketitle

\begin{abstract}\setlength{\parindent}{0pt}\setlength{\parskip}{1ex}\noindent
We show that the convolution product $\ell^1(G)\times \ell^1(G)\to \ell^1(G)$ is not an open map if $G$ is either non-abelian or contains an element of infinite order. The case $G=\Z$ solves a problem posed by Kania and Ma\'slany.
\end{abstract}

\section{Introduction}

Let $A$ be a complex Banach algebra. Since the multiplication map $A\times A\to A$ is bilinear, it does not need to satisfy the open mapping theorem. In fact, the literature contains many recorded examples of Banach algebras with non-open multiplication maps. This motivated the authors in \cite{DrKa18} to study the class of Banach algebras whose multiplication map is (uniformly) open. Such a property has interesting connections with the theory of differential subalgebras \cite{KaMa24} and with the property of stable rank one \cite[Proposition 4.1]{DrKa18}.

Some examples of Banach algebras with an open multiplication map include algebras of functions with bounded variation \cite{KoTu19,CaKaSh21} and other Banach function algebras \cite{BaMaWa13}. Interestingly, in the case of $C(X)$, the openness of the multiplication map is completely characterized in terms of the covering dimension of $X$ \cite{Ko06}, and it holds precisely when ${\rm dim}(X)< 1$.

On the other hand, the literature also contains many examples of algebras whose multiplication is non-open. Among these, we can mention matrix algebras \cite{Be17}, convolution algebras associated with (non-group) cancellative monoids, and algebras of bounded operators over certain classes of Banach spaces \cite{DrKa18}.

For us in this article, the main object of study will be the \emph{convolution algebra} associated with a group $G$. Recall that, given a group $G$, its convolution algebra is the Banach space $\ell^1(G)$, equipped with the \emph{convolution product}:
$$
\Big(\sum_{g\in G}a_g\delta_g\Big)\Big(\sum_{g\in G}b_h\delta_h\Big)=\sum_{k\in G}\Big(\sum_{gh=k}a_gb_h\Big)\delta_k.
$$
This is a Banach algebra under the usual norm $\|\sum_{g\in G}a_g\delta_g\|_1=\sum_{g\in G}|a_g|$. 

The openness of the convolution product has been studied before. Indeed, in \cite{DrKa18}, the authors studied the case of convolution algebras over abelian groups and managed to show that the convolution map is not open if the covering dimension of the dual group is at least two, leaving the case of $\ell^1(\Z)$ open. Furthermore, the openness of multiplication in $\ell^1(\Z)$ was formally stated as an open problem in \cite{KaMa24}. Kania and Ma\'slany also went beyond the case of $G=\Z$ and asked the following question.

\begin{que}[{Kania-Ma\'slany \cite{KaMa24}}]\label{theque}
    Is there an infinite group $G$ for which $\ell^1(G)$ has open multiplication map?
\end{que}

Our main result here shows that, under very mild assumptions on the group, the answer to this question is negative. To be more precise, we prove the following.

\begin{thm}\label{main}
    Let $G$ be a group and suppose that one of the following is true: \begin{enumerate}
        \item $G$ is non-abelian,
        \item $G$ contains an element of infinite order.
    \end{enumerate}
    Then $\ell^1(G)$ does not have an open multiplication map.
\end{thm}

Of course, our main theorem reduces Question \ref{theque} to the case of infinite torsion abelian groups. Their convolution algebras could still have an open or even uniformly open multiplication map. However, we will not speculate about the validity of such a statement.

Before we move into proving Theorem \ref{main}, let us briefly mention that the proof is divided into two steps. In the abelian case, we establish a homotopical obstruction inside the convolution algebra of the ultrapower group $G^\U$. In the non-abelian case, we use the fact that $\ell^1(G)$ is directly finite and contains elements $a,b$ satisfying $ab=0$ and $ba\not=0$. In both cases, we establish concrete examples of elements $(a,b)\in\ell^1(G)\times \ell^1(G)$ that witness the non-openness of the multiplication map.

\section{Proof of the main result}
\subsection{The case of abelian groups}

For a discrete abelian group $G$ and an element $x\in G$ of infinite order, identify the Wiener algebra
$$
 A(\T)=\left\{f(z)=\sum_{k\in\Z}a_kz^k:\|f\|_A=\sum_k|a_k|<\infty\right\}\subset C(\T),
$$
with the closed subalgebra $\ell^1(\langle x\rangle)\subseteq\ell^1(G)$ through the isometric algebra embedding
$$
 f(z)=\sum_k a_kz^k\mapsto\sum_k a_k\delta_{kx}.
$$
Note that we are using additive notation for the group operation in $G$.

\begin{thm}\label{main2}
Let $G$ be a discrete abelian group containing an element $x$ of infinite order. Convolution on $\ell^1(G)$ is not an open map. More precisely, it fails to be open at every pair of nonzero elements $a,b\in\ell^1(\langle x\rangle)$ with $ab=0$.
\end{thm}

In particular, taking $G=\Z$ gives non-openness at every nonzero zero-product pair in the Wiener algebra.

\begin{lem}\label{atomic}
Let $G$ be any group. Fix a non-principal ultrafilter $\mathcal U$ on $\mathbb N$ There is a contractive algebra homomorphism
$$
 Q:\ell^1(G)^{\mathcal U}\rightarrow\ell^1(G^{\mathcal U}),\qquad Q([x_n]_{n\to\U})([k_n]_{n\to\U})=\lim_{\mathcal U}x_n(k_n).
$$
\end{lem}

\begin{proof}[Proof of Theorem~\ref{main2}]
Let $f,g\in A(\T)\setminus\{0\}$ satisfy $fg=0$, and choose
$$
 0<r<\|f\|_\infty,\qquad 0<s<\|g\|_\infty,
$$
viewing $f,g$ in $\ell^1(G)$ through the embedding above. If the product $B(f,r)B(g,s)$ contained a neighborhood of zero, some fixed $c>0$ would allow factorizations
$$
 F_nG_n=c\delta_{nx},\qquad \norm{F_n-f}_1<r,\qquad \norm{G_n-g}_1<s,
$$
valid for all $n\geq 1$. Fix a nonprincipal ultrafilter $\mathcal U$. Put
$$
 U=Q([F_n]_{n\to\U}),\qquad V=Q([G_n]_{n\to\U}),\qquad \nu=[nx]_{n\to \U}\in G^{\mathcal U}.
$$
Then
\begin{equation*}
 UV=c\delta_\nu,\qquad \norm{U-f}_1\le r,\qquad \norm{V-g}_1\le s.
\end{equation*}
Choose $\eta>0$ small enough such that
$$
 r+\eta<\| f\|_\infty,\qquad s+\eta<\| g\|_\infty,
 \qquad \eta(\|U\|_1+\|V\|_1)<|c|/2.
$$
Take finite coefficient truncations $U_0,V_0$ of $U,V$ with errors less than $\eta$. In particular,
\begin{equation}\label{truncation}
 \norm{U_0-f}_1<r+\eta,\qquad
 \norm{V_0-g}_1<s+\eta,\qquad
 \norm{U_0V_0-c\delta_\nu}_1<|c|/2.
\end{equation}
For the last estimate, we used
$$
 U_0V_0-UV=(U_0-U)V_0+U(V_0-V)\qquad \text{and} \qquad\norm{V_0}_1\le\norm{V}_1.
$$

Let $K\le H$ be generated by standard $x$, $\nu$, and the finite supports of $U_0,V_0$. The images of $x,\nu$ in $(K/\operatorname{Tor}K)\otimes_{\Z}\mathbb Q$ are linearly independent. Indeed, a relation $a\nu+bx\in\operatorname{Tor}K$ would give $k(an+b)x=0$ on a set in $\mathcal U$ for some $k\ge1$. Since $x$ has infinite order, this forces $a=b=0$.

As $K$ is finitely generated abelian, rational linear algebra and clearing denominators gives a homomorphism
$$
 \Theta:K\longrightarrow\Z^2,\qquad
 \Theta(x)=(q,0),\quad\Theta(\nu)=(0,m),\qquad qm\ne0.
$$
The induced contractive algebra homomorphism, followed by Fourier transformation, sends $U_0,V_0$ to Laurent polynomials $P,R$ on $\T^2$. The estimates \eqref{truncation} become
\begin{align}
 \sup_{z,w\in\T}|P(z,w)-f(z^q)|&<r+\eta, \\
 \sup_{z,w\in\T}|R(z,w)-g(z^q)|&<s+\eta, \\
 \sup_{z,w\in\T}|P(z,w)R(z,w)-cw^m|&<c/2.\label{productclose}
\end{align}
In particular, we see that 
$$
|P(z,w)Q(z,w)|\geq |cw^n|-|P(z,w)R(z,w)-cw^m|\geq c-c/2>0,
$$
which implies that $P$ and $R$ are nowhere zero on $\T^2$.

Since $q\ne0$, choose $z_f,z_g\in\T$ with $|f(z_f^q)|>r+\eta$ and $|g(z_g^q)|>s+\eta$. The loops $P(z_f,\cdot)$ and $R(z_g,\cdot)$ lie in disks excluding zero, so both have winding number zero. Nonvanishing and connectedness make each winding number independent of $z$. Thus $PR$ has winding number zero in $w$, whereas \eqref{productclose} makes it homotopic through nonvanishing loops to $cw^m$, whose winding number is $m\ne0$. 
\end{proof}

\begin{ex}
    We record a concrete example of a non-zero pair $f,g\in A(\T)$ such that $fg=0$. Indeed, take
$$
 f(e^{it})=\max\{\cos t,0\},\qquad g(e^{it})=\max\{-\cos t,0\}.
$$
They belong to $A(\T)$, as follows from the absolutely convergent series
$$
 |\cos t|=\frac2\pi+\frac4\pi\sum_{n=1}^\infty\frac{(-1)^{n-1}\cos(2nt)}{4n^2-1},\qquad f=\frac{|\cos t|+\cos t}{2},\quad g=\frac{|\cos t|-\cos t}{2}.
$$
In particular, the product of the open balls of radius $1/2$ about these two functions contains zero but no neighborhood of zero.

Equivalently, an explicit pair $a,b\in\ell^1(\Z)$ is given by
$$
 a_0=b_0=\frac1\pi,\qquad a_{\pm1}=\frac14,\qquad b_{\pm1}=-\frac14,
$$
$$
 a_{\pm2n}=b_{\pm2n}=\frac{(-1)^{n-1}}{\pi(4n^2-1)},
$$
with all other coefficients zero. Then $ab=0$ and
$$
 B_{\ell^1(\Z)}(a,1/2)B_{\ell^1(\Z)}(b,1/2)
$$
is not a neighborhood of zero. 
\end{ex}

\subsection{The case of non-abelian groups}

Recall that a unital ring is called \emph{directly finite} if $uv=1$ implies $vu=1$.

\begin{lem}\label{reversible}
Let $A$ be a directly finite unital Banach algebra. If multiplication is open at a pair $(a,b)$ with $ab=0$, then $ba=0$.
\end{lem}
\begin{proof}
For every $i\in \N$, the image of $U_i=B_A(a,1/i)\times B_A(b,1/i)$ is an open neighborhood of zero. Hence, we can choose sequences $a_j\to a$, $b_j\to b$, and nonzero scalars $c_j\to0$ such that $a_jb_j=c_j1$. Direct finiteness applied to $a_j$ and $c_j^{-1}b_j$ gives $b_ja_j=c_j1$. Passing to the limit yields $ba=0$.
\end{proof}

\begin{cor}
    Let $n\geq 2$. Then $M_n(\CC)$ does not have an open multiplication map.
\end{cor}
\begin{proof}
    The algebra $M_n(\CC)$ is directly finite. On the other hand, it contains a copy of $M_2(\CC)$, and the matrices 
    $$
a=\begin{pmatrix}1&0\\0&0\end{pmatrix},
\qquad
b=\begin{pmatrix}0&0\\1&0\end{pmatrix}.
$$
satisfy $ab=0$ but $ba\not=0$.
\end{proof}

We now return to convolution algebras. It is a very classical observation that the algebra $\ell^1(G)$ is directly finite: its faithful left regular representation embeds it into the group von Neumann algebra $\operatorname{VN}(G)$, which is directly finite; see \cite[pag. 122]{Ka72}. An analogous statement is true about the matrix amplifications $M_n(\ell^1(G))$. We observe that the previous corollary immediately yields the non-openness of the multiplication map of $M_n(\ell^1(G))$. This is in sharp contrast to the case of the usual convolution algebras.

We now proceed to prove the remaining part of Theorem~\ref{main}.

\begin{thm}\label{nonabelian}
Let $G$ be a non-abelian group. Then, convolution on $\ell^1(G)$ is not an open map. 
\end{thm}
\begin{proof}
By Lemma~\ref{reversible}, it suffices to find $a,b\in\ell^1(G)$ with $ab=0$ and $ba\ne0$.

Choose $x,h\in G$ with $xh\ne hx$ and put
$$
 R=\{(m,n)\in\Z^2:x^m h x^n=h\}.
$$
This is a subgroup of $\Z^2$, and $(1,-1)\notin R$. Characters of the finitely generated abelian group $\Z^2/R$ separate points. Consequently, there are $z,w\in\T$ such that $z^m w^n=1$ for all $(m,n)\in R$ and $z\ne w$. Moreover,
$$
 x^m h x^n=x^{m'}h x^{n'}\quad\Longleftrightarrow\quad (m-m',n-n')\in R.
$$
Hence, the formula
$$
\Phi(\delta_y)=
\begin{cases}
z^m w^n, & y=x^m h x^n \text{ for some } m,n\in\mathbb Z,\\
0, & y\notin \langle x\rangle h\langle x\rangle.
\end{cases}
$$
defines a well-defined bounded linear functional on $\ell^1(G)$ of norm one.

Now we choose smooth functions $p,q$ on $\T$ with disjoint supports and $p(w)=q(z)=1$. By smoothness, they have absolutely summable Fourier coefficients, and the expressions
$$
 a=\sum_{n\in\Z}\widehat p(n)\delta_{x^n},\qquad b=\sum_{m\in\Z}\widehat q(m)\delta_{x^m}
$$
converge in $\ell^1(G)$. Set $c=b\delta_h$. Since $pq=0$, we have $ab=0=ac$. On the other hand, absolute convergence gives
$$
 \Phi(ca)=\sum_{m,n}\widehat q(m)\widehat p(n)z^m w^n=q(z)p(w)=1,
$$
so $ca\ne0$. Lemma~\ref{reversible} now shows that multiplication is not open at $(a,c)$.
\end{proof}

\section*{AI statement}

The main result was obtained through iterated conversations with GPT-6 Astra. The author then proceeded to simplify and polish the arguments and carefully write this manuscript.

\section*{Acknowledgments}

The author gratefully acknowledges support from the Simons Foundation Dissertation Fellowship SFI-MPS-SDF-00015100. He also thanks Ben Hayes for his support and guidance.

\printbibliography

\bigskip
\bigskip
\newpage
ADDRESS

\smallskip
\smallskip
Felipe I. Flores

Department of Mathematics, University of Virginia,

114 Kerchof Hall. 141 Cabell Dr,

Charlottesville, Virginia, United States

E-mail: hmy3tf@virginia.edu
\end{document}